\documentclass{amsart} 

\usepackage{amsmath,amsthm}     
\usepackage{amssymb,empheq}            
\usepackage{euscript}           
\usepackage{graphicx,enumerate,calc,lscape,color}
\definecolor{red}{rgb}{1,0,0}
\usepackage[matrix,arrow,curve,frame]{xy}    
\usepackage{hyperref}

\xymatrixcolsep{1.9pc}                          
\xymatrixrowsep{1.9pc}
\newdir{ >}{{}*!/-5pt/\dir{>}}                  

\newtheorem{thm}[subsection]{Theorem}
\newtheorem{defn}[subsection]{Definition}
\newtheorem{prop}[subsection]{Proposition}

\newtheorem{cor}[subsection]{Corollary}
\newtheorem{lemma}[subsection]{Lemma}

\theoremstyle{definition}  

\newtheorem{remark}[subsection]{Remark}

  {\end{list}}

\newcommand{\dfn}{\textbf} 

\newcommand{\mdfn}[1]{\dfn{\mathversion{bold}#1}} 

\newcommand{\Smash}             {\wedge}

\newcommand{\tens}              {\otimes}               
\newcommand{\iso}               {\cong}

\newcommand{\cat}{\EuScript}    

\newcommand{\cF}{{\cat F}}

\newcommand{\cP}{{\mathcal P}}

\newcommand{\cW}{{\cat W}}

\newcommand{\Ab}{{\cat Ab}}

\newcommand{\field}[1]  {\mathbb #1} 

\newcommand{\N}         {\field N}

\newcommand{\Z}         {\field Z}

\DeclareMathOperator*{\colim}{colim}

\DeclareMathOperator{\Spec}{Spec}
\DeclareMathOperator{\Hom}{Hom}
\DeclareMathOperator{\uHom}{\underline{Hom}}

\DeclareMathOperator{\id}{id}

\newcommand{\ra}{\rightarrow}                   
\newcommand{\lra}{\longrightarrow}              
\newcommand{\la}{\leftarrow}                    
\newcommand{\llra}[1]{\stackrel{#1}{\lra}}      

\newcommand{\inc}{\hookrightarrow}              

\newcommand{\blank}{-}                          

\newcommand{\und}{\underline}

\newcommand{\rea}[1]{|{#1}|}             

\newcommand{\ceck}[1]{\Cech(#1)}         
\newcommand{\oceck}[1]{\Cech^{o}(#1)}    
\newcommand{\oreal}[1]{\rea{\oceck{U}}}  
\newcommand{\creal}[1]{\rea{\ceck{U}}}   

\newcommand{\Cech}{\check{C}}

\numberwithin{equation}{subsection}

\newenvironment{myequation}
  {\addtocounter{subsection}{1}\begin{eqnarray}}
  {\end{eqnarray}$\!\!$}

\newcommand{\nosee}[1]{}

\newcommand{\DGA}{{\mathcal{D}\mathcal{G}\mathcal{A}}}
\newcommand{\on}{\ \text{on}\ }
\newcommand{\sbt}{{\scriptscriptstyle{\bullet}}}
\newcommand{\us}{{\underline{s}}}
\newcommand{\ut}{{\underline{t}}}

\begin{document}

\title{Grothendieck groups of complexes with null-homotopies}

\author{Daniel Dugger}
\address{Department of Mathematics\\ University of Oregon\\ Eugene, OR
97403} 

\email{ddugger@uoregon.edu}


\maketitle

\tableofcontents

\section{Introduction}
This paper is an addendum to \cite{FH}.  We reprove the main theorem
of that paper by using Grothendieck groups of modules over certain
DGAs.  While the changes to the argument in \cite{FH} are largely
cosmetic, our approach shortens the proof from 25 pages down to 10 and
greatly clarifies the overall structure.  The authors of
\cite{FH} describe their proof as both ``tedious'' and ``cumbersome'',
whereas the approach given here is neither of these.

\subsection{Statement of the main results}
Let $R$ be a commutative ring, and let $S\subseteq R$ be a
multiplicative system.  Let $\cP(R)$ denote the category of bounded chain
complexes $P_\sbt$ such that $P_i=0$ for $i<0$ and where each $P_i$ is
a finitely-generated projective.  
 Define the relative $K$-group $K(R\on S)$ as
follows:

\begin{defn}
\label{de:RonS}
Let $\cF(R\on S)$ be the free abelian group on the isomorphism classes
of objects $P_\sbt$ in $\cP(R)$ having the property that
 $S^{-1}P_\sbt$ is exact.  Let
$K(R\on S)$ be the quotient of $\cF(R\on S)$ by the following relations:
\begin{enumerate}[(1)]
\item $[P_\sbt]=[P'_\sbt]+[P''_\sbt]$ whenever $0\ra P'_\sbt\ra P_\sbt \ra P''_\sbt\ra 0$ is a short
exact sequence;
\item $[P_\sbt]=0$ if $H_*(P_\sbt)=0$.  
\end{enumerate}
\end{defn}

It is easy to see that if one alters the above definition to
use complexes that are nonzero in negative dimensions (but are still
bounded) then this gives rise to an isomorphic group.  In this paper
we will make the convention that all of our chain complexes vanish in
negative dimensions, but this is only for simplicity of presentation.

For $n\geq 0$ define $\cP(R,n)$ to be the full subcategory of $\cP(R)$
consisting of complexes having $P_i=0$ for $i\notin [0,n]$.  Define 
$K(R \on S, n)$ in
a similar way to Definition~\ref{de:RonS} but using this smaller
collection of complexes.  
It is convenient to allow $n=\infty$
here, so that $K(R\on S)=K(R\on S,\infty)$.  
Notice that for $k\leq n\leq \infty$ the inclusion $\cP(R,k)\inc
\cP(R,n)$ induces a map $K(R\on S,k) \ra
K(R\on S,n)$.  

\begin{thm}[Foxby-Halvorsen]
\label{th:FH1}
The map $K(R\on S,k) \ra K(R\on S,n)$ is an
isomorphism for
$1\leq k\leq n\leq \infty$.
\end{thm}

Foxby and Halvorsen actually prove a more general version of the above
result that is necessary for their
applications.  Let $\und{S}=(S_1,\ldots,S_d)$ be a $d$-tuple of
multiplicative systems in $R$.  Say that a chain complex $C_\sbt$ is
\mdfn{$\und{S}$-exact} if $S_i^{-1}C_\sbt$ is exact for every $i$.  
Define $K(R\on \und{S},n)$ just as for $K(R\on S,n)$, except requiring
the complexes $P_\sbt$ to all be $\und{S}$-exact.  

\begin{thm}[Foxby-Halvorsen]
\label{th:FH2}
Let $\und{S}=(S_1,\ldots,S_d)$ be a $d$-tuple of multiplicative
systems.
If $d\leq k\leq n\leq \infty$ then $K(R\on \und{S},k) \ra K(R\on \und{S},n)$ is an
isomorphism.
\end{thm}

From the perspective of algebraic geometry we are looking at
chain complexes of vector bundles on $\Spec R$ that are exact on a Zariski
open subset.  If every open subset had the form $\Spec S^{-1}R$ then
one could be content with Theorem~\ref{th:FH1}, but this is not the
case.  What is true instead is that every open subset of $\Spec R$ is
a finite union of opens of the form $\Spec S^{-1}R$, and therein lies
the importance of Theorem~\ref{th:FH2}.

It seems worth pointing out that the lower bound on $k$ in
Theorem~\ref{th:FH2} is the best possible.  Let $F$ be a field and let
$R=F[x_1,\ldots,x_d]_{(x_1,\ldots,x_d)}$.  Let $S_i=\{x_i^r\,|\, r\geq
0\}$.  The New Intersection Theorem says that all $\und{S}$-exact
complexes in $\cP(R,d-1)$ are exact, and so $K(R\on S,d-1)=0$.  But a
standard argument shows that $K(R\on S)\iso \Z$ (use the Resolution
Theorem to reduce to the Grothendieck group of finite length
$R$-modules, and then use d\'evissage).

\begin{remark}
The notation $K(R\on S)$ is used in \cite{W} and was inspired by the
notation for relative $K$-groups in \cite{TT}.  Because we will have
various Grothendieck groups flying around in the rest of the paper,
and because the choice between denoting these as ``$G$'' or ``$K$''
becomes awkward,  we will just write everything with a ``$G$''.  In particular, we
use $G(R\on S,n)$ for the group denoted $K(R\on S,n)$ above.
\end{remark}

\smallskip

\subsection{Outline of the argument}
\label{se:outline}
For any $s\in R$ let $T_s$ be the unique DGA whose underlying graded algebra is
$R[e]$, where $\deg e=1$, and whose differential satisfies $de=s$.  Of
course the differential is completely described by the formula
\[ d(e^n)=\begin{cases} 0 & \text{if $n$ is even}, \\
se^{n-1} & \text{if $n$ is odd}.
\end{cases}
\]  
Note that for any $t\in R$ there is a map of DGAs $T_{st} \ra T_s$
that sends $e$ to $te$.  

If $C$ is a chain complex of $R$-modules, then giving a left $T_s$-module
structure on $C$ (that extends the $R$-module structure) is equivalent to
specifying a null-homotopy for the multiplication-by-$s$ map $C\ra C$.
This is just the statement that a map of $R$-DGAs 
$T_s\ra \uHom(C,C)$ is determined by the image of $e$, which must be  an
element $\tilde{e}\in\uHom(C,C)_1$ such that
$d\tilde{e}=s\cdot \id_C$.
In particular, a $T_1$-structure on a chain complex $C$ is
the same as a contracting homotopy for $C$.

If $S$ is a multiplicative system, let $tS$ be the category whose
objects are the elements of $S$ and where the set of morphisms from
$s$ to $t$ is $\{x\in S\,|\, sx=t\}$.  This is the so-called {\it
translation category} of the monoid $S$ (under multiplication).
Observe that $tS$ is a filtered category, and that we have defined a
functor $T\colon tS^{op} \ra \DGA$ sending $s\mapsto T_s$.

Write $\cP(T_s,n)$ for the category of dg-modules over $T_s$ whose
underlying chain complex lies in $\cP(R,n)$.  As always, we abbreviate
$\cP(T_s)=\cP(T_s,\infty)$.  

\begin{defn}
\label{de:G(Ts)}
Let $\cF(T_s,n)$ be the free abelian group on isomorphism classes of
objects in $\cP(T_s,n)$.  
Define $G(T_s,n)$ to be the quotient of $\cF(T_s,n)$ by the following
relations:
\begin{enumerate}[(1)]
\item $[X]=[X']+[X'']$ whenever $0\ra X'\ra X \ra X''\ra 0$ is a short
exact sequence in $\cP(T_s,n)$;
\item $[X]=0$ if $H_*(X)=0$.  
\end{enumerate}
\end{defn}

The DGA maps $T_{st}\ra T_s$ yield restriction functors $\cP(T_s,n)\ra
\cP(T_{st},n)$.  These are clearly exact, and so induce maps on Grothendieck groups
\[ j_{st\la s}\colon G(T_s,n) \ra G(T_{st},n).
\]
We will usually drop the subscripts and just write all such maps as ``$j$''.
These maps assemble to define a functor $tS \ra \Ab$
given by $s\mapsto G(T_s,n)$.  

A $T_s$-module is, in particular, a chain complex in $\cP(R,n)$.  
The null-homotopy of the
multiplication-by-$s$ map shows that this complex becomes exact upon
localization at $S$.   So we have the evident homomorphism $G(T_s,n)
\ra G(R\on S,n)$.  Obviously this gives a map $\colim_{tS} G(T_s,n)
\ra G(R\on S,n)$.  A key step in proving the  Foxby-Halvorsen results
will be the following:

\begin{prop}
\label{pr:main1}
For any $1\leq n \leq \infty$, 
the map $\colim_{tS} G(T_s,n) \ra G(R\on S,n)$ is an isomorphism.
\end{prop}

Let $i_{s,n} \colon \cP(T_s,{n}) \ra \cP(T_s,n+1)$ be the evident
inclusion, and use the same notation for the induced map $G(T_s,{n})
\ra G(T_s,n+1)$.  Let $\N$ denote the category whose objects are the
natural numbers $n\geq 1$ and where there is a unique map from $n$ to
$n+1$, for every $n$.  The $i_{s,n}$ maps are clearly compatible
with the change-of-rings maps induced by $T_{st}\ra T_s$, and so
putting them together we have a diagram $tS\times \N \ra \Ab$
sending $(s,n)\mapsto G(T_s,n)$.  Here is a partial depiction of this large diagram:

\[ \xymatrixrowsep{1pc}\xymatrix{
& {}\phantom{ABCDE}\ar[d] & {}\phantom{ABCDEFG}\ar[d] \\
\cdots &  G(T_s,n)\ar[d]^j \ar[l] & G(T_s,n-1)\ar[d]^j\ar[l]^{i_{s,n}} &\cdots \ar[l]\\
\cdots &  G(T_{st},n)\ar[d]\ar[l] & G(T_{st},n-1)\ar[d]\ar[l]^{i_{st,n}} &\cdots \ar[l] \\
& {} & {} \\
& \Downarrow & \Downarrow  \\
\cdots & G(R\on S,n) \ar[l] & G(R\on S,n-1)\ar[l]^{i_n}&\cdots \ar[l] \\
\save "1,2"."4,2"*[F]\frm{}\restore 
\save "1,3"."4,3"*[F]\frm{}\restore 
}
\]
The two framed boxes each depict a slice $tS \ra \Ab$ corresponding
to fixing a value of $n$ (note that because of typographical
constraints we have not drawn a realistic
model for $tS$---this category need not be linear).
 The colimit of these slices is written underneath them.

The second key component of our proof is the following:

\begin{prop}
\label{pr:main2}
For $n\geq 2$ there exist group homomorphisms $\cW_n\colon G(T_s,n)\ra
G(T_{s^2},n-1)$ such that $\cW_n i=j$ and $i\cW_n=j$.  
\end{prop}

The theorem we are after is an immediate corollary of
Propositions~\ref{pr:main1} and \ref{pr:main2}:

\begin{cor}
\label{co:maincor}
For any $n\geq 1$ the map
$G(R\on S,n) \ra G(R\on S,n+1)$ is an isomorphism.  Consequently, the
maps $G(R\on S,n) \ra G(R\on S,\infty)$ are isomorphisms.  
\end{cor}

\begin{proof}
It is trivial that $G(R\on S,\infty)=\colim_n G(R\on S,n)$, and so the
second statement follows from the first.  The proof of the first
statement is formal, but we include it anyway.  Let $\alpha\in G(R\on S,n+1)$.  By
Proposition~\ref{pr:main1} there exists an $s\in S$ and an $\alpha'\in
G(T_s,n+1)$ that maps to $\alpha$.  Then $\cW_{n+1}(\alpha')$ yields
an element in $G(R\on S,n)$ that must map to $\alpha$ because
$i\cW_{n+1}=j$.  This shows that $i_n$ is surjective.  Injectivity
is similar: if $\beta\in G(R\on S,n)$ is in the kernel, then by
Proposition~\ref{pr:main1} there exists $s\in S$ for which we can lift
$\beta$ to a $\beta'\in G(T_s,n)$ that is in the kernel of
$i_{s,n}$.  Then
$0=\cW_{n+1}(0)=\cW_{n+1}i_{s,n}(\beta')=j(\beta')$.  
Since $\beta'$ and $j\beta'$ map to the same element 
in $G(R\on S,n)$, namely $\beta$, we must have $\beta=0$.  
\end{proof}

We have now explained how the first main result (Theorem~\ref{th:FH1})
is an immediate consequence
of Propositions~\ref{pr:main1} and \ref{pr:main2}.  The story is
slightly trickier than we have indicated, however: a key insight is that
Proposition~\ref{pr:main2} is actually {\it used to prove\/}
Proposition~\ref{pr:main1}.  So in reality the structure of the
argument is:
\begin{itemize}
\item Prove Proposition~\ref{pr:main2},
\item Use Proposition~\ref{pr:main2} to prove
Proposition~\ref{pr:main1}, and finally
\item Use Propositions~\ref{pr:main1} and \ref{pr:main2} to deduce Theorem~\ref{th:FH1}.
\end{itemize}

Everything we have said so far concerns the proof of
Theorem~\ref{th:FH1}, which is merely the $d=1$ case of
Theorem~\ref{th:FH2}.  However, the proof has been structured so that
it adapts almost verbatim to the case of general $d$.  Details are
in Section~\ref{se:general}.

\begin{remark}
At the risk of seeming repetitive we again point out that the basic
idea for the proof is entirely due to \cite{FH}.  In particular, that paper
gives a version of the maps $\cW_n$.  The contribution of the present
paper is the repackaging of their construction into the above outline,
which simplifies many technicalities.
\end{remark}

\section{Preliminary material}
\label{se:prelim}
This section establishes some basic observations, constructions, and notation.

\begin{prop}
\label{pr:Ts-existence}
Let $P_\sbt$ be a bounded chain complex of finitely-generated projective
$R$-modules such that $S^{-1}P_\sbt$ is exact.  Then there exists
$s\in S$ such that $P_\sbt$ has the structure of a $T_s$-module.
\end{prop}

\begin{proof}
For  $R$-modules $M$ and $N$ there is a natural map $S^{-1}\Hom_R(M,N) \ra
\Hom_{S^{-1}R}(S^{-1}M,S^{-1}N)$.  This is clearly an isomorphism when
$M=R$, and both the functors $S^{-1}\Hom_R(\blank,N)$ and
$\Hom_{S^{-1}R}(S^{-1}(\blank),S^{-1}N)$ commute with finite direct
sums and are left-exact.  It follows that the natural map is an
isomorphism for all finitely-presented modules---and in particular,
for all finitely-generated projectives.  

It follows from the preceding paragraph that  
the canonical map of chain complexes $S^{-1}\uHom_R(P_\sbt,P_\sbt)\ra
\uHom_{S^{-1}R}(S^{-1}P_\sbt,S^{-1}P_\sbt)$ is an isomoprhism.  This uses that
$P_\sbt$ is bounded and that its modules are finitely-generated.
Since $S^{-1}P_\sbt$ is exact, it follows that $S^{-1}\uHom_R(P_\sbt,P_\sbt)$ is
exact.  In particular, there exists an element $e\in
S^{-1}\uHom(P_\sbt,P_\sbt)_1$  such that $de=\id$.  The element $e$
may be written as $e'/s$ for some $e'\in \uHom(P_\sbt,P_\sbt)_1$ and some $s\in
S$, and we have $de'=s\id$ in $S^{-1}\uHom(P_\sbt,P_\sbt)$.  So there exists a
$u\in S$ such that $u(de'-s\id)=0$ in $\uHom(P_\sbt,P_\sbt)$.  This means
$d(ue')=(us)\cdot \id$, and so $ue'$ gives $P_\sbt$ the structure of a
$T_{us}$-module.
\end{proof}

\begin{prop}
\label{pr:exact-nh}
Let $0\ra A \ra B \ra C \ra 0$ be a short exact sequence of chain
complexes of finitely-generated projective $R$-modules.  If $A$ is a
$T_s$-module and $C$ is a $T_t$-module, then there exists a
$T_{st}$-module structure on $B$ such that the exact sequence is an
exact sequence of $T_{st}$-modules (where $A$ and $C$ become
$T_{st}$-modules via restriction-of-scalars along the maps $T_{st}\ra T_s$
and $T_{st}\ra T_t$).  
\end{prop}

\begin{proof}
Consider the following diagram of chain complexes:
\[ \xymatrix{
\uHom(C,A) \ar[r] \ar[d] & \uHom(B,A) \ar[r] \ar[d] & \uHom(A,A)
\ar[d] \\
\uHom(C,B) \ar[r] \ar[d] & \uHom(B,B) \ar[r] \ar[d] & \uHom(A,B)
\ar[d] \\
\uHom(C,C) \ar[r] & \uHom(B,C) \ar[r] & \uHom(A,C).
}
\]
Every row is a short exact sequence, as is every column.  In the
following argument we will just write $(BC)$ as an abbreviation for
$\Hom(B,C)$, and so forth.  The argument is basically a diagram chase,
but also making use of the fact that $(BA)$ is a left-module over
$(AA)$.

We have elements $e\in (CC)_1$ such that $de=t\cdot \id$ and $f\in
(AA)_1$ such that $df=s\cdot \id$.  Our goal is to produce an element
$g\in (BB)_1$ such that $dg=st\cdot \id_B$ and such that (1) both $g$
and $se$ map to the same element of $(BC)_1$, and (2) both $g$ and $tf$
map to the same element of $(AB)_1$.  This will prove the proposition.

Let $e'\in (CB)_1$ be a preimage
for $e$, and let $e''$ be the image of $e'$ in $(BB)_1$.  Write
$a=de''$.  One sees readily that both $a$ and $t\cdot\id_B$ map to
the same element of $(BC)_0$, and so $a-t\cdot \id_B$ is the image of an
element $b\in (BA)_0$.  This $b$ must be a cycle, since
$a-t\cdot \id_B$ is.

Next, consider the element $f\cdot b\in (BA)_1$.  This satisfies
$d(fb)=sb$ by the Leibniz rule.  If $x$ is the image of $fb$ in
$(BB)_1$ then $dx=sa-st\id_B$, and so $d(se''-x)=st\cdot \id_B$.  Let
$z=se''-x$.  The last thing to do is to calculate the image of $z$ in
$(AB)$; this is the same as the image of $-x$, which by commutativity
of the upper left square is the same as the image of $-f\cdot b$ in
$(AB)$.

Let $b'$ be the image of $b$ in $(AA)$.  Both $b'$ and $-t\cdot \id_A$
map to $-t\cdot i$ under $(AA)\ra(AB)$, where $i$ is $A\inc B$.  
So $b'=-t\cdot \id_A$ by injectivity.  
The image of $f\cdot b$ in $(AA)$ is then $f\cdot b'=-tf$.
So  $z$ and $tf$ map to the same element in $(AB)$.
\end{proof}

\begin{remark}
\label{re:cone}
A common use of Proposition~\ref{pr:exact-nh} occurs when $X$
is a $T_s$-module, $Y$ is a $T_t$-module, and $f\colon X\ra Y$ is a
map of chain complexes over $R$.  Let $Cf$ be the mapping cone of $f$,
which sits in a short exact sequence $0\ra Y \ra Cf \ra \Sigma X\ra
0$.  The above proposition guarantees us a $T_{st}$-module structure
on $Cf$ so that this is an exact sequence of $T_{st}$-modules.  In
this case it is useful to be completely explicit about what the module
structure is.  To this end, write $(Cf)_n=Y_n\oplus \sigma X_{n-1}$,
where the $\sigma$ is simply a marker that will be useful in
formulas.  For example, the differential on $Cf$ is given by
$d_{Cf}(y)=dy$ for $y\in Y_n$ and $d_{Cf}(\sigma x)=f(x)-\sigma dx$ for
$x\in X_{n-1}$.  The $\sigma$-notation is also used in $(\Sigma
X)_n=\sigma X_{n-1}$, where the differential is $d(\sigma
x)=-\sigma(dx)$.  Likewise, the $T_s$-module structure on $\Sigma X$
is $e.(\sigma x)=-\sigma(e.x)$.  If one regards the symbol $\sigma$  as
having degree $1$ and satisfying the usual Koszul sign conventions,
all of the above formulas are easy to remember.

If $E$ denotes the generator for $T_{st}$, define a module
structure on $Cf$ by
\[\begin{cases} E.y=s.ey \text{ for $y\in Y_n$, and}\\
 E.(\sigma
x)=ef(ex)-t.\sigma(ex)\ \text{for $x\in X_{n-1}$}.
\end{cases}
\]
One readily checks that this is indeed a $T_{st}$-module structure, and
that it has the desired properties.

In the case where $s=t$ and $f$ is a map of $T_s$-modules we can do
slightly better.  Here one can put a $T_s$-module structure on $Cf$
via the formulas $E.y=ey$ and $E.(\sigma x)=-\sigma(ex)$.  In other
words, we do not need to pass to $T_{s^2}$ here.  
\end{remark}

\begin{remark}
Here is an observation that will be used often.
If $X$ is an
object in $\cP(T_s,n-1)$ then the cone on the identity map sits in a
short exact sequence $0\ra X \ra C(\id) \ra \Sigma X \ra 0$.  This is
an exact sequence of $T_s$-modules using the last paragraph of
Remark~\ref{re:cone}, and all the modules lie in $\cP(T_s,n)$.  We
obtain $[\Sigma X]=-[X]$ in $G(T_s,n)$.  
\end{remark}

Finally,
consider an $X$ in $\cP(T_s,n)$.  The subcomplex $0\ra X_n
\llra{d} d(X_n) \ra 0$ of $X$ (concentrated in degrees $n$ and $n-1$)
is a $T_s$-submodule, and therefore the quotient of $X$ by this
submodule has an induced
$T_s$-structure.  This quotient is the chain complex
$\tau_{\leq n-1}X$, which coincides with $X$ in degrees less than
$n-1$ and equals $X_{n-1}/d(X_n)$ in degree $n-1$.  The complex
$\tau_{\leq n-1}X$ need not be in $P(T_s,n-1)$, as there is no
guarantee that $X_{n-1}/d(X_n)$ is a projective $R$-module.

However, assume that $X$ is contractible as a chain complex.  Then $X$
admits a structure of $T_1$-module (not related to the
$T_s$-structure), and this consists of a collection of maps
$X_{i-1}\ra X_i$ corresponding to left-multiplication by $e$.  If
$x\in X_n$ then $ex=0$ and so $0=d(ex)=x-e(dx)$, hence $e.dx=x$.  So
the map $X_{n-1}\ra X_n$ is a left inverse for $d\colon X_n\ra
X_{n-1}$, and hence $X_{n-1}/d(X_n)$ is projective.  

For $P$ an $R$-projective let $D^n_s(P)$ be the complex $0\ra P
\llra{\id} P \ra 0$, concentrated in degrees $n$ and $n-1$, with
$T_s$-module structure given by the multiplication-by-$s$ map.
We have proven the following:

\begin{prop}
\label{pr:contract}
Let $X\in \cP(T_s,n)$ and assume the underlying chain complex of $X$ is contractible.  
Then there is a short exact sequence in $\cP(T_s,n)$ of the form
\[ 0 \ra D^n_s(X_n) \llra{f} X \lra \tau_{\leq n-1}X \ra 0
\]
where $f$ equals the identity in degree
$n$ and equals $d$ in degree $n-1$.
\end{prop}

One may continue with the above process, producing a similar exact
sequence with $\tau_{\leq n-1}X$ as the middle term and then
proceeding inductively.  So every contractible $T_s$-module may be
built up via extensions from the complexes $D^k_s(P)$.  This proves the following:

\begin{cor}
\label{co:contract}
If Definition~\ref{de:G(Ts)} is changed so that relation (2) is only imposed
for modules of the form $D_s^k(P)$, $1\leq k\leq n$ and $P$ a
finitely-generated $R$-projective, then this
yields the same group $G(T_s,n)$.  
\end{cor}

\begin{remark}[Duality]
\label{re:duality}
If $X$ is a $T_s$-module then $\Hom(X,R)$ has an induced
$T_s$-structure: one simply turns all the arrows around.  
This gives a duality functor on each of the categories
$\cP(T_s,n)$.  It is useful to remember that every construction in
this paper has a dual version.  For example, this applies to
Proposition~\ref{pr:contract} above.  
\end{remark}


\section{The construction of the $\cW$ maps}

In this section we will construct the maps $\cW\colon G(T_s,n) \ra
G(T_{s^2},n-1)$ for $n\geq 2$ and prove Proposition~\ref{pr:main2}.

\medskip

For $n\geq 2$ we describe a functor $\Gamma\colon \cP(T_s,n)\ra
\cP(T_{s^2},n-1)$.  For $X$ in $\cP(T_s,n)$ this functor takes the top
degree $X_n$ and ``folds'' it down into degree $n-2$, interchanging the
$d$ and $e$ operators on this piece.  This doesn't quite yield a
$T_s$-module, but a little fiddling yields a $T_{s^2}$-structure. For
$s=1$ this construction appears in \cite[2.10]{D}, and a related
construct is in
\cite[proof of Lemma 2.6.9]{A}.  

To avoid some confusion let us  write $e$ for the polynomial generator of $T_s$
and $E$ for the corresponding generator of $T_{s^2}$.  
If $X\in \cP(T_s,n)$ define
\[  (\Gamma X)_i=\begin{cases}
X_i & \text{if $0\leq i\leq n-3$ or $i=n-1$,} \\
X_{n-2}\oplus X_n & \text{if $i=n-2$},\\
0 & \text{otherwise.}
\end{cases}
\]
The differential $(\Gamma X)_i \ra (\Gamma X)_{i-1}$ coincides with
the one on $X$ (denoted $d_X$) for $i\leq n-3$.
For $(a,b)\in X_{n-2}\oplus X_n=(\Gamma X)_{n-2}$ we set
$d(a,b)=d_Xa$, and for $a\in X_{n-1}=(\Gamma X)_{n-1}$ we set
$da=(d_Xa,ea)$.  This clearly makes $\Gamma X$ into a chain complex of
finitely-generated, projective $R$-modules.  To complete the
construction we need to describe the $T_{s^2}$-module structure.  For
$a\in (\Gamma X)_i$ set
\[ E.a=\begin{cases}
s(e.a) & \text{if $0\leq i\leq n-3$}, \\
s(e.a)-e^2.da & \text{if $a\in X_{n-2}\subseteq (\Gamma X)_{n-2}$},\\
s.da & \text{if $a\in X_{n}\subseteq (\Gamma X)_{n-2}$} \\
0 & \text{otherwise.}
\end{cases}
\]
It is routine, although slightly tedious, to check that this really is
the structure of a $T_{s^2}$-module (for $u\in X_{n-2}$ use that
$0=d(e^3u)=se^2u-e^3du$).  The fact that $\Gamma$ is a functor is then
self-evident.  Moreover, it is clearly an exact functor: if $0\ra
X'\ra X \ra X''\ra 0$ is a short exact sequence of $T_s$-modules then
$0\ra \Gamma X' \ra \Gamma X \ra \Gamma X''\ra 0$ is still short
exact.  Finally, we note that if $i\colon \cP(T_s,n-1)
\ra \cP(T_s,n)$ is the evident inclusion then the composition
\[ \cP(T_s,n-1) \llra{i} \cP(T_s,n) \llra{\Gamma} \cP(T_{s^2},n) 
\]
is just the restriction-of-scalars functor $j$.

Write $K(s)$ for the Koszul complex $R\llra{s} R$, concentrated in
degrees $0$ and $1$.  It has an evident structure of $T_s$-module, since
$K(s)$ is just $T_s/(e^2)$.  Let us record the following simple calculation:

\begin{lemma}\label{le:gamma1}  
If $P$ is an $R$-module then
$\Gamma(D^n_sP)\iso \Sigma^{n-2}(K(s)\tens_R P)$ (as $T_{s^2}$-modules), for $n\geq 2$.  
\end{lemma}

\begin{remark}[A better approach to $\Gamma X$]
The definition of $\Gamma$ seems to have come out of thin air.  To
understand it better, notice that
for any $X\in \cP(T_s,n)$  there is a natural chain map
$f\colon X\ra \Sigma^{n-1}(K(s)\tens_R X_n)$.  It equals the identity in
degree $n$ and is given (up to sign) by left-multiplication by $e$ in degree
$n-1$.  Note that this is not a map of $T_s$-modules.  However, let
$Cf$ be the mapping cone of $f$.  By Remark~\ref{re:cone} there is a canonical
structure of $T_{s^2}$-module on $Cf$ such that
\begin{myequation}
\label{eq:ex1}
 0 \ra \Sigma^{n-1}(K(s)\tens_R X_n) \ra Cf \ra \Sigma X \ra 0
\end{myequation}
is an exact sequence of $T_{s^2}$-modules. 

Note that $(Cf)_{n+1}=X_n$, and consider the canonical map of $T_{s^2}$-modules
$D^{n+1}_{{s^2}}(X_n) \ra Cf$.  This is an inclusion (because $f$
was the identity in degree $n$) and the image is a
$T_{s^2}$-subcomplex, so we may consider the quotient module.  An
easy calculation shows that this quotient is isomorphic to $\Sigma(\Gamma X)$:
\begin{myequation}
\label{eq:ex2}
 0 \ra D^{n+1}_{s^2}(X_n) \ra Cf \ra \Sigma(\Gamma X) \ra 0.
\end{myequation}

Each of the complexes in both (\ref{eq:ex1}) and (\ref{eq:ex2}) are
 concentrated in degrees $1$ through $n+1$; so if we desuspend
 everything we may regard these as exact sequences in
 $\cP(T_{s^2},n)$.
Since $D^{n+1}_{s^2}(X_n)$ is contractible we have
$[\Sigma^{-1}Cf]=[\Gamma X]$ in $G(T_{s^2},n)$, and (\ref{eq:ex1})
then gives the identity
\begin{myequation}
\label{eq:ex3}
\qquad\quad [X]=[\Sigma^{-1}Cf]-[\Sigma^{n-2}K(s)\tens_R X_n]=[\Gamma
X]+(-1)^{n-1}[K(s)\tens_R X_n]
\end{myequation}
in $G(T_{s^2},n)$.  
\end{remark}

We are ready to construct the $\cW$ maps and prove they have the
desired properties:

\begin{proof}[Proof of Proposition~\ref{pr:main2}]
For every $X\in \cP(T_s,n)$ let
\[ \cW_n(X)=[\Gamma X] - [\Sigma^{n-2}K(s)\tens_R X_n] \in G(T_{s^2},n-1).\]
This assignment is additive since both terms in the definition of $\cW_n(X)$ are
additive. Using Corollary~\ref{co:contract}, the assignment will descend to the
Grothendieck group if we verify $\cW_n(X)=0$ for $X=D^k_s(P)$, where $1\leq
k\leq n$ and $P$ is a finitely-generated projective.  If $k<n$ this is
trivial because then $\Gamma X=X$ (which is contractible) and
$X_n=0$.  For $k=n$ the claim follows from Lemma~\ref{le:gamma1}.

So $\cW_n$ induces
a group homomorphism $G(T_s,n)\ra G(T_{s^2},n-1)$, which we will also call
$\cW_n$.
The identity $\cW_n i=j$ is an immediate consequence of $\Gamma i=j$,
and
 $i\cW_n=j$ is simply (\ref{eq:ex3}).
\end{proof}

\section{Proof of the colimit result}

In this section we prove Proposition~\ref{pr:main1}.

\begin{lemma}
\label{le:colim1}
Let $P_\sbt\in \cP(T_s,n)$, and suppose that there is another
$T_s$-module structure on the same underlying chain complex of
$R$-modules.  Denote the two $T_s$-modules as $P^1_\sbt$ and
$P^2_\sbt$. 
Then the images of $[P^1_\sbt]$ and $[P^2_\sbt]$ in $\colim_s
G(T_s,n)$ are the same.
\end{lemma}

\begin{proof}
Start by considering the short
exact sequence $0\ra P_\sbt \ra C(\id)\ra \Sigma P_\sbt \ra 0$ where
the middle term is the cone on the identity map.  The second two
complexes in the sequence have length $n+1$.  Equip the left term 
$P_\sbt$ with the first $T_s$-module structure, and equip the right
term $\Sigma P_\sbt$ with the second $T_s$-module structure
(corresponding to $\Sigma P^2_\sbt$).  By
Proposition~\ref{pr:exact-nh} there is a $T_{s^2}$-module stucture on
$C(\id)$ such that the above is an exact sequence of
$T_{s^2}$-modules.  Since $C(\id)$ is contractible we therefore have
\[ [P_\sbt^1]=-[\Sigma P^2_\sbt]=[P^2_\sbt] \]
in $G(T_{s^2},n+1)$.  It is perhaps better to write this as $j_{s^2\la
s}i_s[P_\sbt^1]=j_{s^2\la s}i_s[P_\sbt^2]$, where $i_s$ is the map
$G(T_s,n) \ra G(T_s,n+1)$.
By commuting the $j$ and $i$ we then get $i_{s^2}j_{s^2\la
s}[P^1_\sbt]=i_{s^2}j_{s^2\la s}[P^2_\sbt]$.  

Now apply the map $\cW_{s^2}\colon G(T_{s^2},n+1) \ra G(T_{s^4},n)$ to
get 
\[ \cW_{s^2} i_{s^2}j_{s^2\la s} [P^1_\sbt]=
\cW_{s^2} i_{s^2}j_{s^2\la s} [P^2_\sbt].
\]
By Proposition~\ref{pr:main2} one has $\cW_{s^2}i_{s^2}=j_{s^4\la
s^2}$, and so the above identity simply says
$j_{s^4\la s}[P^1_\sbt]=j_{s^4\la s}[P^2_\sbt]$.
Consequently,  $[P^1_\sbt]=[P^2_\sbt]$ in $\colim_s
G(T_s,n)$.  
\end{proof}

\begin{proof}[Proof of Proposition~\ref{pr:main1}]
We must prove that the map $j\colon \colim_s G(T_s,n) \ra G(R\on S,n)$ is an
isomorphism, and we will do this by constructing an inverse. 
Let $P_\sbt$ be a  bounded complex of finitely-generated projectives that is
$S$-exact.  Then by Proposition~\ref{pr:Ts-existence} $P_\sbt$ may be
given the structure of $T_s$-module, for some
$s\in S$.  By Lemma~\ref{le:colim1} the class of this $T_s$-module in
$\colim_s G(T_s,n)$ is independent of the choice of $T_s$-module
structure: write this class as $\{P\}$.  

We must show that the assignment $P_\sbt \mapsto \{P\}$ is additive
and sends contractible complexes to zero.  The latter is trivial.  For
the former, let $0\ra P_\sbt' \ra P_\sbt \ra P''_\sbt \ra 0$ be an
exact sequence in $\cP(R\on S,n)$.  Choose a $T_s$-module structure on
$P_\sbt'$ and on $P''_\sbt$.  By Proposition~\ref{pr:exact-nh} there is a
$T_{s^2}$-module structure on $P_\sbt$ making the above an exact
sequence of $T_{s^2}$-modules.  So we have
$[P_\sbt]=[P'_\sbt]+[P''_\sbt]$ in $G(T_{s^2},n)$, and this
immediately yields $\{P\}=\{P'\}+\{P''\}$.  So we have constructed a
map $\lambda\colon G(R\on S,n)\ra \colim_s G(T_s,n)$.  The equation
$j\lambda =\id$ is obvious, and the equation $\lambda j=\id$ is
immediate from the fact that the definition of $\lambda$ did not
depend on the choice of $T_s$-module structure.
\end{proof}


\section{Proof of the general Foxby-Halvorsen result}
\label{se:general}

Fix $d\geq 1$ and let $\us\in R^n$.  Define $T_\us$ to be the free DGA
over $R$ generated by elements $e_1,\ldots,e_d$ of degree $1$
satisfying $d(e_i)=s_i$.  As an algebra $T_\us$ is the free (non-commutative)
$R$-algebra on $d$ variables.    Note that to give a $T_\us$-module
structure on a chain complex $C_\sbt$ is the same as giving
null-homotopies for the multiplication-by-$s_i$ maps with $1\leq
i\leq d$.  Because the null-homotopies are independent of each other,
the analogs of the results in Section~\ref{se:prelim} are all
straightforward consequences of what we have already proven.

If $\us,\ut\in R^n$ write $\us\ut$ for the tuple whose $i$th
element is $s_i\cdot t_i$.  There are canonical DGA maps $T_{\us\ut}
\ra T_{\us}$ sending $e_i\mapsto t_ie_i$.  If $S_1,\ldots,S_d$ are
multiplicative systems then these maps assemble to give a diagram of
DGAs $(tS_1\times \cdots \times tS_d)^{op} \ra \DGA$.  
Note that $tS_1\times \cdots \times tS_d$
 is  a product of filtered categories, hence filtered.  

Let $\cP(T_\us,n)$ and $G(T_\us,n)$ be defined in the evident way,
generalizing our notation from Section~\ref{se:outline}.  We obtain a
diagram of abelian groups $tS_1\times \cdots \times tS_d \times \N \ra
\Ab$,
with $(s_1,\ldots,s_d,n)\mapsto G(T_\us,n)$.  We again use $j$ for any
map in this diagram between objects with $n$ fixed, and $i$ for any
map between objects with $\us$ fixed.  

\begin{thm}\label{th:general}\mbox{}\par
\begin{enumerate}[(a)]
\item For $n\geq d+1$ 
there are maps $\cW\colon G(T_\us,n)\ra G(T_{\us^2},n-1)$ satisfying
$\cW i=j$ and $i\cW=j$.  
\item For $n\geq d$ the maps $\colim_{\us} G(T_\us,n) \ra G(R\on \und{S},n)$ are
isomorphisms.
\item The maps $G(R\on \und{S},n)\ra G(R\on \und{S},n+1)$ are
isomorphisms, for all $n\geq d$.  
\end{enumerate}
\end{thm}

The proof of the above result closely parallels what we did for
$d=1$.  The main component is the construction of a functor
$\Gamma\colon P(T_\us,n)\ra P(T_{\us^2},n-1)$ that ``folds'' the
degree $n$ piece down into the lower degrees.  The only difference
from $d=1$ is that this functor is a bit more complicated.  To see
why, note that the maps $e_1,\ldots,e_d\colon X_{n-1}\ra X_n$ force us
to now fold $d$ different copies of $X_n$ into degree $n-2$.  The
$e_i$ maps on these folded terms will all be the old differential
$d_X$, but one quickly finds that this does not give anything close to
a $T_{\us^2}$-structure.   It turns out that in order to fix this one
needs to put certain ``interaction'' terms into the lower
degrees---these interaction terms end up forming a higher-dimensional
Koszul complex (see Remark~\ref{re:final} for a picture).  

We start by establishing some notation.  For $\us\in
R^n$ let $K(\us)$ denote the Koszul complex for $s_1,\ldots,s_d$.
This is just the quotient of $T_\us$ by the relations
$e_ie_j+e_je_i=0$ and $e_i^2=0$, $1\leq i,j\leq d$.  In particular,
note that $K(\us)$ becomes a $T_\us$-module in the evident way.

\begin{lemma}
Let $X$ be a $T_\us$-module.  There there is a chain map $K(\us)\tens_R X_0
\ra X$  that is the identity in degree $0$, and is natural in $X$.  
(Warning: This is not a
map of $T_\us$-modules).
\end{lemma}

\begin{proof}
For $1\leq i_1<\cdots<i_k\leq d$ and $u\in X_0$, send
$e_{i_1}\cdots e_{i_k}\tens u \mapsto e_{i_1}\cdots e_{i_k}.u$.
The Leibniz rule and the fact that $du=0$ readily show this to be
a chain map.
\end{proof}

What we actually need is the dual of the above lemma: if $X\in
\cP(T_s,n)$ then there is a natural chain map $f\colon X\ra
\Sigma^{n-d}(\Omega(\us)\tens_R X_n)$ where $\Omega(\us)$ is the dual
of the Koszul complex.  (As it happens, the Koszul complex is
self-dual---but we do not need this fact.)  This dual result follows
immediately, using Remark~\ref{re:duality}.  But
in the interest of beng concrete let us describe the chain map $f$
more precisely.  This description is unnecessary for our application,
but it seems worth including.

Let $e_1,\ldots,e_d$ be the standard basis for $R^d$, so that
$\us=\sum_j s_je_j$.  Define the complex
$\Omega(\us)$ by $\Omega(\us)_i=\Smash^{n-i}R^d$, where
the differential $d\colon \Omega(\us)_i\ra \Omega(\us)_{i-1}$ sends
$\omega\mapsto \omega\Smash \us$.  The complex $\Omega(\us)$ is the
dual of $K(\us)$, but in fact we also have 
$K(\us)\iso \Omega(\us)$.  The isomorphism is via the
Hodge $*$-operator: if $\omega=e_{i_1}\Smash \cdots\Smash e_{i_k}$
define $*\omega=e_{j_1}\Smash\cdots \Smash e_{j_{n-k}}$ to be the
unique wedge product of this form having the property that
$\omega\Smash *\omega=(-1)^{k(n-k)}e_1\Smash\cdots\Smash e_n$.  One
readily checks that the assignment $\omega\mapsto *\omega$ is a chain
map $K(\us)\ra \Omega(\us)$, and then it is clearly an isomorphism.
We use this observation only to see that $\Omega(\us)$ has a
$T_\us$-structure.  [The structure has $e_r\cdot
(e_{j_1}\Smash\cdots\Smash e_{j_t})$ equal to zero if $e_r$ does not
appear in the wedge product, and equal to $(-1)^{a+t}e_{j_1}\Smash
\cdots \widehat{e_{j_a}} \Smash\cdots \Smash e_{j_t}$ if $j_a=r$.] 

\begin{lemma}
\label{le:FH-main}
Let $X$ be a $T_\us$-module such that $X_i=0$ for $i>n$.  Then there
is a chain map $X \ra \Sigma^{n-d}(X_n\tens \Omega(\us))$ that equals
the identity in degree $n$, and this map is natural in $X$.   (Warning:
this is not a map of $T_\us$-modules).
\end{lemma}

\begin{proof}
For any $u\in X_{n-k}$ and any $j_1,\ldots,j_{k+1}\in [1,d]$ the Leibniz rule
yields 
\addtocounter{subsection}{1}
\begin{align}
\label{eq:leibniz}
 0=d(0)& =d(e_{j_1}\cdots e_{j_{k+1}}.u)\\ \notag
&=\sum_{a=1}^{k+1} (-1)^{a-1}
s_{j_a}\cdot e_{j_1}\cdots \widehat{e_{j_a}}\cdots e_{j_{k+1}}.u
+ (-1)^{k+1} e_{j_1}\cdots e_{j_{k+1}}.du.
\end{align}
Define $f\colon X_{n-k} \ra X_n\tens \Omega(\us)_k$ by the formula
\[ u\mapsto \sum_{i_1<\cdots<i_k} (e_{i_1}e_{i_2}\cdots
e_{i_k}.u)\tens (e_{i_1}\Smash\cdots\Smash e_{i_k}).
\]
The verification that this is a chain map is routine, using (\ref{eq:leibniz}).
\end{proof}

\begin{proof}[Proof of Theorem~\ref{th:general}]
Once (a) is established, the proofs of (b) and (c) are the same
as the $d=1$ case we have already done.  So the only work is in
proving (a).

Let $X$ be an object in $\cP(T_\us,n)$, and let $f\colon X \ra
\Sigma^{n-d}(X_n\tens \Omega(\us))$ be the map provided by
Lemma~\ref{le:FH-main}.  Let $Cf$ be the mapping cone of $f$, which
comes equipped with a canonical $T_{\us^2}$-module structure.  The
fact that $f$ equals the identity in the top degree yields that $Cf$
has the complex $D^{n+1}_{s^2}(X_n)$ as a subcomplex.
Indeed, one readily checks that this is a sub-$T_{\us^2}$-module.  Let
$QX$ denote the quotient, so that we have two short exact sequences of
$T_{\us^2}$-modules
\addtocounter{subsection}{1}
\begin{align}
\label{eq:final}
 & 0\ra \Sigma^{n-d}(X_n\tens \Omega(\us)) \ra Cf \ra \Sigma X \ra 0, 
\quad\text{and}\quad\\
& 0 \ra D^{n+1}_{s^2}(X_n) \ra Cf \ra QX \ra 0.\notag
\end{align}
Note that $QX$ is concentrated in degrees $1$ through $n$ (this uses
that $n\geq d+1$), and let $\Gamma X$ be the desuspension of $QX$.  

It is obvious that $\Gamma X$ is functorial in $X$, and it is also
obvious that $\Gamma(\blank)$ is additive (one only has to note that
in each degree $\Gamma X$ is a direct sum of certain homogeneous
components of $X$).  Define 
\[ \cW_n(X)=[\Gamma X] + [\Sigma^{n-d}(X_n\tens
\Omega(\us))] \in G(T_{\us^2},n).
\]
The proof that this descends to a map $G(T_\us,n) \ra G(T_\us,n-1)$ is
exactly the same as the $d=1$ case---it boils down to the very simple computation
that $\Gamma(D^n_{\us}P)\iso
\Sigma^{n-d}(K(\us)\tens_R P)$ as $T_{\us^2}$-modules.  
It is evident that $\cW i=j$.  The
identity $i\cW=j$ follows immediately from the two short exact
sequences in (\ref{eq:final}) once one realizes that the complexes all
lie in degrees $1$ through $n+1$, and therefore the sequences may be
desuspended to give exact sequences in $\cP(T_{\us}^2,n)$.  This proves (a).
\end{proof}

\begin{remark}
\label{re:final}
The complex $\Gamma X$ looks as shown below, where the
differential has both the horizontal components and the cross-term
depicted diagonally:
\[ \xymatrixcolsep{1pc}\xymatrix{
X_0 & X_1\ar[l]  & \cdots\ar[l]  
& X_{n-d-1}\ar[l]\ar@{}[d]|\oplus  & \cdots\ar[l]\ar[dl]  &
X_{n-3}\ar[l]\ar[dl] \ar@{}[d]|\oplus 
& X_{n-2}\ar[l]\ar[dl]\ar@{}[d]|\oplus 
& X_{n-1}\ar[l]\ar[dl] \\
&&& X_n\tens \Smash^d R^d & \cdots\ar[l] & X_n\tens \Smash^2 R^d\ar[l]
& 
X_n\tens \Smash^1 R^d\ar[l]
}
\]
\end{remark}

We are finally ready to complete our proof of the main result:

\begin{proof}[Proof of Theorem~\ref{th:FH2}]
This is a formal consequence of Theorem~\ref{th:general}; see the proof of
Corollary~\ref{co:maincor}.
\end{proof}

\bibliographystyle{amsalpha}

\end{document}